\documentclass[11pt,a4paper]{article}

\usepackage[margin=1in]{geometry}
\usepackage{palatino}
\usepackage{subcaption}

\usepackage{enumitem}

\usepackage[
  colorlinks,
  linkcolor = blue,
  citecolor = blue,
  urlcolor = blue]{hyperref}

\usepackage{amsmath,amssymb,amsthm}
\usepackage{mathtools}
\mathtoolsset{centercolon}

\usepackage[affil-it]{authblk}
\usepackage{tabu}
\usepackage{blkarray}

\makeatletter
\newtheorem*{rep@theorem}{\rep@title}
\newcommand{\newreptheorem}[2]{%
\newenvironment{rep#1}[1]{%
 \def\rep@title{#2 \ref{##1}}%
 \begin{rep@theorem}}%
 {\end{rep@theorem}}}
\makeatother

\newreptheorem{theorem}{Theorem}
\newreptheorem{lemma}{Lemma}

\newtheorem{theorem}{Theorem}[section]
\newtheorem*{theorem*}{Theorem}
\newtheorem{lemma}[theorem]{Lemma}
\newtheorem{cor}[theorem]{Corollary}

\theoremstyle{definition}
\newtheorem{definition}[theorem]{Definition}
\newtheorem{remark}[theorem]{Remark}
\newtheorem{example}[theorem]{Example}

\newcommand{\N}{\mathbb{N}}
\newcommand{\C}{\mathbb{C}}
\newcommand{\Z}{\mathbb{Z}}
\newcommand{\F}{\mathbb{F}}

\DeclareMathOperator{\aut}{Aut}

\title{Quantum isomorphic strongly regular graphs from the $E_8$ root system}
\author{Simon Schmidt\footnote{sisc@math.ku.dk}}
\affil{QMATH, Department of Mathematical Sciences, University of Copenhagen, Universitetsparken 5, 2100 Copenhagen \O, Denmark}
\date{\today}

\begin{document}
\maketitle
\begin{abstract}
In this article, we give a first example of a pair of quantum isomorphic, non-isomorphic strongly regular graphs, that is, non-isomorphic strongly regular graphs having the same homomorphism counts from all planar graphs. The pair consists of the orthogonality graph of the $120$ lines spanned by the $E_8$ root system and a rank $4$ graph whose complement was first discovered by Brouwer, Ivanov and Klin. Both graphs are strongly regular with parameters $(120, 63, 30, 36)$. Using Godsil-McKay switching, we obtain more quantum isomorphic, non-isomorphic strongly regular graphs with the same parameters. 
\end{abstract}

\section{Introduction}
The notion of quantum isomorphism of graphs was first introduced by Aterias et al in \cite{AMRSSV}. Two graphs are quantum isomorphic if there exists a perfect quantum strategy for the so called isomorphism game, a nonlocal game in which Alice and Bob want to convince a referee that they know an isomorphism between two graphs $G_1$ and $G_2$. Interestingly, Aterias et al constructed pairs of graphs for which there are perfect quantum strategies for the isomorphism game, but no perfect classical strategies. This shows that there are graphs that are quantum isomorphic, but not isomorphic. 

It is well-known that two graphs are isomorphic if there exists a permutation matrix interchanging the adjacency matrices of the graphs. It was shown in \cite{AMRSSV} that two graphs are quantum isomorphic if and only if there exists a quantum permutation matrix $u$ interchanging the adjacency matrices of the two graphs. Here a quantum permutation matrix is a matrix $u\in M_n(\mathcal{A})$ with entries in some unital $C^*$-algebra $\mathcal{A}$ fulfilling $u_{ij}=u_{ij}^*=u_{ij}^2$ and $\sum_k u_{ik}=1_{\mathcal{A}}=\sum_k u_{ki}$, which was first studied in terms of quantum permutation groups by Wang \cite{WanSn}.
A completely combinatorial description of quantum isomorphism was given by Man\v{c}inska and Roberson in \cite{MRplanar}. They showed that quantum isomorphism is equivalent to having the same homomorphism counts from all planar graphs. 

Despite having many equivalent formulations, only a few constructions of quantum isomorphic, non-isomorphic graphs are known. In \cite{AMRSSV}, the graphs where constructed from quantum solutions of binary constraint systems. Roberson and the author (\cite{RS}) used colored versions of the graphs constructed in \cite{AMRSSV} and a decoloring procedure to obtain new quantum isomorphic, non-isomorphic graphs, but those still come from quantum solutions of binary constraint systems. 
A more general approach was presented by Musto, Reutter and Verdon in \cite{MRV}. In the article, they obtain quantum isomorphic graphs from a central type subgroup of the automorphism group of one of the graphs, having coisotropic stabilizers. They could find the graphs of \cite{AMRSSV} in this way, but they did not construct new examples. Recently, Chan and Martin \cite{CM} and Gromada \cite{G} have shown that Hadamard graphs of the same order are quantum isomorphic. 

In this article, we will construct a new pair of quantum isomorphic, non-isomorphic graphs from a subgroup of the automorphism group of a graph. This subgroup is associated to $3$-tensor Pauli matrices and is a central type subgroup with coisotropic stabilizers. We will explicitly construct the quantum permutation matrix interchanging the adjacency matrices from the automorphisms in the subgroup. The quantum isomorphic graphs are both strongly regular with parameters $(120, 63, 30, 36)$ and known from the graph theory literature. The first graph is the orthogonality graph of the lines in the $E_8$ root system, see also \cite[Section 10.39]{BVM}. The second graph can be obtained from independent sets of the folded halved $8$-cube graph. Its complement was first discovered by Brouwer, Ivanov and Klin in \cite{BIK} as a graph from a quadric with a hole. 
One can obtain one graph from the other by switching the underlying partial geometry, as shown by Mathon and Street \cite{MS}. Note that our first graph is isomorphic to $G_1$ and the second graph is isomorphic to $G_2$ -- $G_5$ in \cite[Table 2.2]{MS}. 

We will also show that using Godsil-McKay switching on both graphs with the same, appropriate vertex partition preserves quantum isomorphism. We get more examples of quantum isomorphic, but non-isomorphic strongly regular graphs in this way. 

The article is structured as follows. In Section \ref{secprelim}, we recall basic notions of graph theory and give the definition of quantum isomorphism of graphs. In Section \ref{secpairqiso}, we will first define the orthogonality graph of the lines in the $E_8$ root system and study a specific subgroup of the automorphism group. Then, we look at a graph coming from independent sets of the folded halved $8$-cube graph and give an alternative description of this graph. Using this description, we show that the two mentioned graphs are quantum isomorphic. Finally, we obtain more quantum isomorphic, non-isomorphic strongly regular graphs using Godsil-McKay switching in Section \ref{secswitching}.

\section{Preliminaries}\label{secprelim}
We start by recalling basic notions of graph theory. In this article, a graph $G$ is always finite and simple, that is, it has a finite vertex set $V(G)$ and has no multiple edges or loops. Thus, the edge set $E(G)$ is a subset of $V(G)\times V(G)$, where $(i,j)\in E(G)$ implies $(j,i)\in E(G)$. The \emph{adjacency matrix} $A_G$ of a graph $G$ is the matrix with entries $(A_G)_{ij}=1$ if $(i,j)\in E(G)$ and $(A_G)_{ij}=0$ otherwise. The complement of a graph $G$, which we denote by $\overline{G}$, is the graph with the same vertex set as $G$ and $(i,j)\in E(\overline{G})$ if and only if $(i,j)\notin E(G)$. 
A \emph{clique} is a subset $C\subseteq V(G)$ of vertices such that all vertices in $C$ are adjacent. The \emph{clique number} of $G$ is the size of a largest clique in $G$. An \emph{independent set} is a subset $I\subseteq V(G)$ of vertices such that all vertices in $I$ are non-adjacent. The $\emph{independence number}$ of $G$ is the size of a largest independent set in $G$.

The vertex $j$ is called \emph{neighbor} of $i$ if $(i,j)\in E(G)$. A graph is \emph{$k$-regular} if every vertex has $k$ neighbors. A \emph{path} of length $m$ joining two vertices $i,k \in V(G)$ is a sequence of vertices $a_0, a_1, \dots, a_m$ with $i=a_0$, $k=a_m$ such that $(a_n,a_{n+1})\in E(G)$ for $n\in \{0, \dots m-1\}$. The \emph{distance} $d(i,k)$ between vertices $i,k\in V(G)$ denotes the length of a shortest path joining $i$ and $k$. 

\begin{definition}
Let $G$ be a $k$-regular graph with $n$ vertices. We say that the graph $G$ is \emph{strongly regular} if there exist $\lambda, \mu\in \N_0$ such that
\begin{itemize}
    \item[(i)] adjacent vertices have $\lambda$ common neighbors,
    \item[(ii)] non-adjacent vertices have $\mu$ common neighbors. 
\end{itemize}
We then say that $G$ is strongly regular with parameters $(n, k, \lambda, \mu)$.
\end{definition}

A \emph{graph automorphism} is a bijection $\sigma:V(G)\to V(G)$ such that $(i,j)\in E(G)$ if and only if $(\sigma(i), \sigma(j))\in E(G)$. The set of graph automorphisms form a group, the \emph{automorphism group} $\aut(G)$. For a subgroup $K\subseteq \aut(G)$ and a vertex $v\in V(G)$, we call $\mathrm{Stab}_K(v)=\{\sigma \in K \,|\, \sigma(v)=v\}$ the \emph{stabilizer subgroup} of $K$ with respect to $v$.

An \emph{isomorphism} between graphs $G_1$ and $G_2$ is a bijection $\varphi:V(G_1) \to V(G_2)$ such that $(i,j)\in E(G_1)$ if and only if $(\varphi(i), \varphi(j))\in E(G_2)$. It is easy to see that there exists an isomorphism between the graphs $G_1$ and $G_2$ if and only if there exists a permutation matrix $P_{\varphi}$ such that $A_{G_1}P_{\varphi}=P_{\varphi}A_{G_2}$. 

\begin{definition}[\cite{WanSn}]
Let $\mathcal{A}$ be a unital $C^*$-algebra. A matrix $u\in M_n(\mathcal{A})$ is called \emph{quantum permutation matrix} or \emph{magic unitary} if the entries $u_{ij}\in \mathcal{A}$ are projections, i.e. $u_{ij}=u_{ij}^*=u_{ij}^2$ for all $i,j$ and 
\begin{align*}
    \sum_k u_{ik}=1_{\mathcal{A}}=\sum_k u_{ki}
\end{align*}
for all $i$. 
\end{definition}
If $\mathcal{A}=\C$, then $u$ is a quantum permutation matrix if and only if it is a permutation matrix. For the quantum permutation matrices appearing in this article, we will have $\mathcal{A}=M_8(\C)$. 

The concept of quantum isomorphism was first introduced in \cite{AMRSSV} as perfect quantum strategies of the isomorphism game. We will use an equivalent definition, established in \cite{LMR}. 

\begin{definition}
Let $G_1$ and $G_2$ be graphs. We say that $G_1$ and $G_2$ are \emph{quantum isomorphic} if there exists a unital $C^*$-algebra $\mathcal{A}$ and a quantum permutation matrix $u\in M_n(\mathcal{A})$ such that $A_{G_1}u=uA_{G_2}$, which means $\sum_k(A_{G_1})_{ik}u_{kj}=\sum_k u_{ik}(A_{G_2})_{kj}$ for all $i\in V(G_1), j \in V(G_2)$.
\end{definition}

The following lemma gives an equivalent relation to $A_{G_1}u=uA_{G_2}$, see also \cite[Theorem 2.2 and Theorem 2.5]{LMR}. We give a proof similar to \cite[Proposition 2.1.3]{Sthesis} which gives the same equivalence for quantum automorphism groups. 

\begin{lemma}\label{lemprod0}
Let $G_1$ and $G_2$ be graphs, $\mathcal{A}$ be a unital $C^*$-algebra and $u\in M_n(\mathcal{A})$ be a quantum permutation matrix. Then $A_{G_1}u=uA_{G_2}$ is equivalent to $u_{ij}u_{kl}=0$ if $(i,k)\notin E(G_1)$ and $(j,l)\in E(G_2)$ or vice versa.
\end{lemma}

\begin{proof}
Assume $A_{G_1}u=uA_{G_2}$ and let $(i,k)\notin E(G_1)$, $(j,l)\in E(G_2)$. From $A_{G_1}u=uA_{G_2}$, we get $\sum_{s;(i,s)\in E(G_1)}u_{sl}=\sum_{t:(t,l)\in E(G_2)} u_{it}$. Using this, we calculate
\begin{align*}
    u_{ij}u_{kl}= u_{ij}\left(\sum_{t:(t,l)\in E(G_2)} u_{it}\right)u_{kl}=u_{ij}\left(\sum_{s;(i,s)\in E(G_1)}u_{sl}\right)u_{kl}=0.
\end{align*}
The first equality holds because we have $u_{ij}u_{it}=\delta_{jt}u_{ij}$ and $j$ is part of the sum since $(j,l)\in E(G_2)$. The last equality holds because $u_{sl}u_{kl}=\delta_{ks}u_{kl}$ and $k$ is not part of the sum since $(i,k)\notin E(G_1)$. One similarly gets $u_{ij}u_{kl}=0$ for $(i,k)\in E(G_1), (j,l)\notin E(G_2)$. 

Now assume that we have $u_{ij}u_{kl}=0$ if $(i,k)\notin E(G_1)$ and $(j,l)\in E(G_2)$ or vice versa. Using this and $\sum_s u_{sl}=1_{\mathcal{A}}=\sum_t u_{it}$, we get
\begin{align*}
    \sum_{t;(t,l)\in E(G_2)}u_{it}&=\left(\sum_{t;(t,l)\in E(G_2)}u_{it}\right)\left(\sum_{s;(i,s)\in E(G_1)}u_{sl}\right)\\
    &=\sum_{s;(i,s)\in E(G_1)}\left(\sum_{t=1}^nu_{it}\right)u_{sl}\\
    &=\sum_{s;(i,s)\in E(G_1)}u_{sl}
\end{align*}
for all $i,l$. 
\end{proof}

Interestingly, quantum isomorphism of graphs is connected to homomorphism counts of planar graphs in the following way. 

\begin{theorem}\cite{MRplanar}
Let $G_1$ and $G_2$ be graphs. Then $G_1$ and $G_2$ are quantum isomorphic if and only if they admit the same number of homomorphisms from any planar graph. 
\end{theorem}

As mentioned in the introduction, only few constructions of quantum isomorphic, non-isomorphic graphs are known. In the next section, we obtain a new pair of quantum isomorphic, non-isomorphic graphs.

\section{A pair of quantum isomorphic strongly regular graphs}\label{secpairqiso}

In this section, we construct a new pair of quantum isomorphic, non-isomorphic graphs. Those graphs are additionally strongly regular with parameters $(120, 63, 30, 36)$, giving us a first pair of quantum isomorphic strongly regular graphs. Both graphs are known from the graph theory literature. We will first define the graphs, give some properties and alternative descriptions and then prove that they are quantum isomorphic. 

\subsection[Orthogonality graph]{The orthogonality graph of the lines in the $E_8$ root system}

\begin{definition}\label{defGE8}
The $E_8$ root system $\Psi_{E_8}$ consists of the following $240$ vectors in $\mathbb{R}^8$:
\begin{align}\label{rootsystemE8}
&\pm e_i \pm e_j \text{ for } 1 \leq i<j \leq 8,\qquad  x=(x_1,\dots,x_8) \text{ for } x_i \in \{\pm 1\} \text{ and } \prod_{i=1}^8 x_i=1.
\end{align}
Let $G_{E_8}=(V(G_{E_8}), E(G_{E_8}))$ be the orthogonality graph of the vectors in \eqref{rootsystemE8}, where we identify vectors $x$ and $-x$. This means that $G_{E_8}$ is the graph with $V(G_{E_8})=\{v_x \,|\, x \in \Psi_{E_8}, x \equiv-x\}$, $(v_x,v_y)\in E(G_{E_8})$ if and only if $\langle x, y\rangle=0$. This graph is a strongly regular with parameters $(120, 63, 30, 36)$, see for example \cite[Section 10.39]{BVM}. \end{definition}

Note that the graph $G_{E_8}$ is isomorphic to the graph $G_1$ appearing in \cite[Table 2.2]{MS}. We will now look at some automorphisms of $G_{E_8}$.
We denote by $I$, $X$, $Y$ and $Z$ the (real Pauli) matrices 
\begin{align}\label{Pauli}
I=\begin{pmatrix}1&0\\0&1\end{pmatrix},\quad
X=\begin{pmatrix}0&1\\1&0\end{pmatrix},\quad Z=\begin{pmatrix}1&0\\0&-1\end{pmatrix} \quad\text{and } Y=XZ.
\end{align}

\begin{lemma}\label{Lsubgroup}
Let $I$, $X$, $Y$ and $Z$ be as above. For each $M:=M_1\otimes M_2 \otimes M_3$ with $M_i \in \{I,X,Y,Z\}$ the maps $\sigma_M:V(G_{E_8})\to V(G_{E_8})$, $v_x \mapsto v_{Mx}$ are automorphisms of $G_{E_8}$. Those automorphisms give rise to a  subgroup $L \cong \Z_2^6$ of $\aut(G_{E_8})$.
\end{lemma}

\begin{proof}
Let $M:=M_1\otimes M_2 \otimes M_3$ for some $M_i \in \{I,X,Y,Z\}$. We first have to check that $\sigma_M$ is well-defined. For this, we have to show that $Mx\in \Psi_{E_8}$ for every $x \in \Psi_{E_8}$. Note that any matrix $M$ of this form is a product of the following six matrices:

\begin{align}
&X\otimes I\otimes I=\begin{pmatrix}&&I&0\\&&0&I\\I&0&&\\0&I&&\end{pmatrix}, 
&&Z\otimes I\otimes I=\begin{pmatrix}I&&&\\&I&&\\&&-I&\\&&&-I\end{pmatrix},\nonumber\\
&I\otimes X\otimes I=\begin{pmatrix}0&I&&\\I&0&&\\&&0&I\\&&I&0\end{pmatrix},
&&I\otimes Z\otimes I=\begin{pmatrix}I&&&\\&-I&&\\&&I&\\&&&-I\end{pmatrix}
,\nonumber\\
&I\otimes I\otimes X=\begin{pmatrix}X&&&\\&X&&\\&&X&\\&&&X\end{pmatrix},
&&I\otimes I\otimes Z=\begin{pmatrix}Z&&&\\&Z&&\\&&Z&\\&&&Z\end{pmatrix}.\label{Paulitensor}
\end{align}
Thus, the matrix $M$ is a signed permutation matrix flipping an even number of signs. Looking at \eqref{rootsystemE8}, we see that such permutations map every $x\in\Psi_{E_8}$ to an element in $\Psi_{E_8}$. Therefore, $\sigma_M$ is well-defined.

Now, we show that every $\sigma_M$ is a graph automorphism. We have the following relations on $X,Y,Z$:
\begin{align*}
    X^2=I, X^*=X, Y^*=-Y, Y^2=-I, Z^2=I, Z^*=Z.
\end{align*}
Knowing that $M$ is a product of the matrices in \eqref{Paulitensor}, this yields $MM^*=M^*M=(I\otimes I\otimes I)$. From this, one deduces that $\sigma_M$ is a bijection. Since unitaries preserve inner products, we get $\langle x,y\rangle=0$ if and only if $\langle Mx,My \rangle=0$ and thus $(v_x,v_y)\in E(G_{E_8})$ if and only if $(v_{Mx}, v_{My})\in E(G_{E_8})$. This shows $\sigma_M\in \aut(G_{E_8})$.

It remains to show that we get a subgroup $L\cong \Z_2^6$ of the automorphism group. It is easy to see that $\sigma_M \sigma_N=\sigma_{MN}$. Also note that $\sigma_M=\sigma_{-M}$, since we identified $x$ and $-x$ in the vertex set of $G_{E_8}$. Thus, products and inverses are still of the form $\sigma_{N_1\otimes N_2\otimes N_3}$ for some $N_i \in \{I,X,Y,Z\}$. The group $L$ is generated by $\sigma_{X\otimes I\otimes I}, \sigma_{I\otimes X\otimes I},\sigma_{I\otimes I\otimes X},\sigma_{Z\otimes I\otimes I},\sigma_{I\otimes Z\otimes I},$ $\sigma_{I\otimes I\otimes Z}$. Since we have $\sigma_M=\sigma_{-M}$ and all matrices in \eqref{Paulitensor} either commute or anticommute, we get a commutative subgroup of $\aut(G_{E_8})$. Since we also have $\sigma_M^2=\mathrm{id}$ for all generators and they are independent, we obtain $L\cong \Z_2^6$.
\end{proof}

The next lemma shows that vertices in the same orbit of $L$ have the same stabilizer. 

\begin{lemma}\label{stabeq}
If there exists $\sigma_M\in L$ with $\sigma_M(v_x)=v_y$, then $\mathrm{Stab}_L(v_x)=\mathrm{Stab}_L(v_y)$.
\end{lemma}

\begin{proof}
Let $\sigma_N \in \mathrm{Stab}_L(v_x)$. Then, we know $v_x=v_{Nx}$ and thus $x =\pm Nx$. Multiplying by $M$ from the left yields $Mx=\pm MNx$. Since $M$ and $N$ either commute or anti-commute, we get $Mx=\pm NMx$. This shows $y=\pm Ny$ and therefore $\sigma_N \in \mathrm{Stab}_L(v_y)$. Thus $\mathrm{Stab}_L(v_x)\subseteq \mathrm{Stab}_L(v_y)$. One similarly shows the other inclusion and we get $\mathrm{Stab}_L(v_x)= \mathrm{Stab}_L(v_y)$.
\end{proof}

Using the action of the subgroup $L\subseteq \aut(G_{E_8})$, we get a partition of the vertex set of $G_{E_8}$ into $15$ orbits. 

\begin{lemma}
The action of $L$ on $V(G_{E_8})$ yields $15$ orbits of size $8$. Those partition the vertex set into $15$ cliques of size $8$. 
\end{lemma}

\begin{proof}
Let $S\subseteq \{1,\dots, 8\}$. We denote by $x_S\in \mathbb{R}^8$ the vector with $(x_S)_i=-1$ if $i\in S$ and $(x_S)_i=1$ otherwise. By straightforward computation, we have the following orbits under $L$, where we just list the vectors $x$ associated to the vertices $v_x$ in the graph $G_{E_8}$:
\begin{align}
   V_1:\quad &e_1 \pm e_2, e_3 \pm e_4, e_5 \pm e_6, e_7 \pm e_8,\nonumber\\
   V_2:\quad &e_1 \pm e_3, e_2 \pm e_4, e_5 \pm e_7, e_6 \pm e_8,\nonumber\\
   V_3:\quad &e_1 \pm e_4, e_2 \pm e_3, e_5 \pm e_8, e_6 \pm e_7,\nonumber\\
   V_4:\quad &e_1 \pm e_5, e_2 \pm e_6, e_3 \pm e_7, e_4 \pm e_8,\nonumber\\
   V_5:\quad &e_1 \pm e_6, e_2 \pm e_5, e_3 \pm e_8, e_4 \pm e_7,\nonumber\\
   V_6:\quad &e_1 \pm e_7, e_2 \pm e_8, e_3 \pm e_5, e_4 \pm e_6,\nonumber\\
   V_7:\quad &e_1 \pm e_8, e_2 \pm e_7, e_3 \pm e_6, e_4 \pm e_5,\nonumber\\
   V_8:\quad &x_{\{1,2\}}, x_{\{3,4\}},x_{\{5,6\}},x_{\{7,8\}},x_{\{1,4,6,8\}}, x_{\{2,3,6,8\}},x_{\{2,4,5,8\}},x_{\{2,4,6,7\}},\nonumber\\
   V_9:\quad  &x_{\{1,3\}}, x_{\{2,4\}},x_{\{5,7\}},x_{\{6,8\}},x_{\{1,4,7,8\}}, x_{\{1,4,5,6\}},x_{\{1,2,6,7\}},x_{\{1,2,5,8\}},\nonumber\\
   V_{10}:\quad  &x_{\{1,4\}}, x_{\{2,3\}},x_{\{5,8\}},x_{\{6,7\}},x_{\{1,3,7,8\}}, x_{\{1,3,5,6\}},x_{\{1,2,5,7\}},x_{\{1,2,6,8\}},\nonumber\\
   V_{11}:\quad  &x_{\{1,5\}}, x_{\{2,6\}},x_{\{3,7\}},x_{\{4,8\}},x_{\{1,6,7,8\}}, x_{\{2,5,7,8\}},x_{\{4,5,6,7\}},x_{\{1,2,4,7\}},\nonumber\\
   V_{12}:\quad  &x_{\{1,6\}}, x_{\{2,5\}},x_{\{3,8\}},x_{\{4,7\}},x_{\{1,5,7,8\}}, x_{\{2,6,7,8\}},x_{\{3,5,6,7\}},x_{\{4,5,6,8\}},\nonumber\\
   V_{13}:\quad  &x_{\{1,7\}}, x_{\{2,8\}},x_{\{3,5\}},x_{\{4,6\}},x_{\{1,5,6,8\}}, x_{\{3,6,7,8\}},x_{\{2,5,6,7\}},x_{\{4,5,7,8\}},\nonumber\\
   V_{14}:\quad  &x_{\{1,8\}}, x_{\{2,7\}},x_{\{3,6\}},x_{\{4,5\}},x_{\{1,5,6,7\}}, x_{\{4,6,7,8\}},x_{\{2,5,6,8\}},x_{\{3,5,7,8\}},\nonumber\\
   V_{15}:\quad  &x_\emptyset, x_{\{5,6,7,8\}},x_{\{3,4,7,8\}},x_{\{2,4,6,8\}},x_{\{3,4,5,6\}}, x_{\{2,4,5,7\}},x_{\{2,3,6,7\}},x_{\{2,3,5,8\}}.\label{vertexpartition}
\end{align}
We see that the vectors in the orbits form a basis of $\mathbb{R}^8$, thus they partition $V(G_{E_8})$ into $15$ cliques of size $8$.
\end{proof}

We have the following stabilizer subgroups for all vertices in the corresponding orbits:
\begin{align}
    &V_1:\,\,\, \langle \sigma_{IIX}, \sigma_{IZI}, \sigma_{ZII}\rangle, &&V_2:\,\,\,\langle \sigma_{ZII}, \sigma_{IXI}, \sigma_{IIZ}\rangle, &&V_3:\,\,\,\langle \sigma_{ZII}, \sigma_{IXX}, \sigma_{IZZ}\rangle,\nonumber\\
    &V_4:\,\,\,\langle \sigma_{XII}, \sigma_{IIZ}, \sigma_{IZI}\rangle, &&V_5:\,\,\,\langle \sigma_{XIX}, \sigma_{IZI}, \sigma_{ZIZ}\rangle, &&V_6:\,\,\,\langle \sigma_{XXI}, \sigma_{IIZ}, \sigma_{ZZI}\rangle,\nonumber\\
    &V_7:\,\,\,\langle \sigma_{XXX}, \sigma_{ZZI}, \sigma_{IZZ}\rangle, &&V_8:\,\,\,\langle \sigma_{IIX}, \sigma_{ZXI}, \sigma_{XZI}\rangle, &&V_9:\,\,\,\langle \sigma_{IXI}, \sigma_{ZIX}, \sigma_{XIZ}\rangle,\nonumber\\
    &V_{10}:\langle \sigma_{IXX}, \sigma_{ZXI}, \sigma_{XZZ}\rangle, &&V_{11}:\langle \sigma_{XII}, \sigma_{IXZ}, \sigma_{IZX}\rangle, &&V_{12}:\langle \sigma_{XIX}, \sigma_{IZX}, \sigma_{ZXZ}\rangle,\nonumber\\
    &V_{13}:\langle \sigma_{XXI}, \sigma_{IXZ}, \sigma_{ZZX}\rangle, &&V_{14}:\langle \sigma_{XXX}, \sigma_{ZZX}, \sigma_{XZZ}\rangle, &&V_{15}:\langle \sigma_{XII}, \sigma_{IXI}, \sigma_{IIX}\rangle.\label{Stabs}
\end{align}
In this case, we use the notation $M_1M_2M_3:=M_1 \otimes M_2 \otimes M_3$, where $M_i\in \{I,X,Y,Z\}$.
Note that the rank-$1$ projections associated to the vectors can be written in the form 
\begin{align*}
    \frac{1}{\|x\|^2}xx^*=\frac{1}{8}(1\pm N_1)(1\pm N_2)(1\pm N_3),
\end{align*}
where $N_i, i=1,2,3$ are the matrices associated to the generators of the stabilizer subgroup. For example, we have
\begin{align*}
    \frac{1}{2}(e_1+e_2)(e_1+e_2)^*=\frac{1}{8}(1+IIX)(1+IZI)(1+ZII).
\end{align*}
Using those rank-$1$ projections, we will now define a quantum permutation matrix on the vertex set of $G_{E_8}$. 

\begin{lemma}\label{qpermmatrix}
Let $P_x=\frac{1}{\|x\|^2}xx^*$ be the associated rank-$1$ projection to every vector $x \in \Psi_{E_8}$ in the $E_8$ root system. We partition the vertex set of $G_{E_8}$ as in \eqref{vertexpartition}. For each $i \in \{1,\dots, 15\}$, we choose a vector $w_i$ associated to one of the vertices in $V_i$. For $v_y, v_z \in V_i$, we define $u_{v_yv_z}^{(i)}:=M_{yz}P_{w_i}M_{yz}^*=P_{M_{yz}w_i}$, where $M_{yz}=M_1\otimes M_2 \otimes M_3$,  $M_i\in \{I,X,Y,Z\}$ fulfils $M_{yz}y=\pm z$ (such an $M$ exists as $v_y$ and $v_z$ are in the same orbit under the action of $L$). Let $u^{(i)}=(u_{v_yv_z}^{(i)})_{v_y,v_z \in V_i}$ and let $u$ be the matrix 
\begin{align}
\begin{blockarray}{cccccc}
&V_1&V_2&\dots&\dots&V_{15}\\
\begin{block}{c(ccccc)}
V_1&u^{(1)}&0&0&\dots&0\\ V_2&0&u^{(2)}&0&\dots&0\\ \vdots&0&0&u^{(3)}&\dots&0\\\vdots&\vdots &\vdots&\vdots&\ddots&0\\V_{15}&0&0&0&0&u^{(15)} \\
\end{block}
\end{blockarray}\label{blockmatrix}
\end{align}
Then $u$ is a quantum permutation matrix.
\end{lemma}

\begin{proof}
The entries of $u$ are projections by definition. It remains to show that $\sum_{v_y\in V_i}u_{v_yv_z}^{(i)}=1_{M_8(\C)}$ and $\sum_{v_z\in V_i}u_{v_yv_z}^{(i)}=1_{M_8(\C)}$ for all $i$. Since $M_{yz}^2=\pm 1_{M_8(\C)}$, we see that we also have $M_{yz}z=\pm y$. For $v_y, v_{z_1}, v_{z_2}\in V_i$ with $z_1\neq z_2$, we thus have $M_{yz_2}M_{yz_1}z_1=\pm z_2$. It follows $\sigma_{M_{yz_2}M_{yz_1}}\notin \mathrm{Stab}_L(v_{z_1})=\mathrm{Stab}_L(v_{w_i})$, where the equality of the stabilizers comes from Lemma \ref{stabeq}. We get $M_{yz_1}w_i\neq \pm M_{yz_2}w_i$ for all $v_{z_1}\neq v_{z_2} \in V_i$. Recall that we have $u_{v_yv_z}^{(i)}=P_{M_{yz}w_i}$. We deduce
\begin{align*}
    \sum_{v_z\in V_i}u_{v_yv_z}^{(i)}=\sum_{v_z\in V_i}P_{M_{yz}w_i}=1_{M_8(\C)},
\end{align*}
since we sum over all rank-$1$ projections of the vectors associated to $V_i$, which form a basis of $\mathbb{R}_8$ (see \eqref{vertexpartition}). One similarly shows $\sum_{v_y\in V_i}u_{v_yv_z}^{(i)}=1_{M_8(\C)}$.
\end{proof}

For the next lemma, recall that $d(i,j)$ denotes the distance between vertices $i,j$ as defined in Section \ref{secprelim}. 

\begin{lemma}\label{edgepermutation}
Partition the vertex set of $G_{E_8}$ as in \eqref{vertexpartition}. Let $k,s \in V_i$ and $l,t \in V_j$ for $i\neq j$. If $d(k,l)=d(s,t)$, then there exists $\sigma \in L$ such that $\sigma(k)=s$, $\sigma(l)=t$.
\end{lemma}

\begin{proof}
Since $k,s \in V_i$, they are in the same orbit under the action of $L$ by definition. Thus, there exists $\sigma_1\in L$ such that $\sigma_1(k)=s$, $\sigma_1(l)=t'$, where $d(k,l)=d(s,t')$. 

Let $\mathrm{Stab}_L(s)\cap \mathrm{Stab}_L(t')=\{\mathrm{id}, \tau\}$ (we know $|\mathrm{Stab}_L(s)\cap \mathrm{Stab}_L(t')|=2$, see \eqref{Stabs}). Assume $\sigma_2, \sigma_3 \in \mathrm{Stab}_L(s)$ such that $\sigma_2(t')=\sigma_3(t')$. Then $\sigma_2\circ \sigma_3\in \mathrm{Stab}_L(s)\cap \mathrm{Stab}_L(t')$ since $\sigma_2\circ \sigma_3(t')=\sigma_2\circ \sigma_2(t')=t'$ as $\sigma_2$ has order $2$. Therefore, we either have $\sigma_2\circ \sigma_3=\mathrm{id}$ or $\sigma_2\circ \sigma_3=\tau$. This shows $\sigma_3=\sigma_2$ or $\sigma_3=\sigma_2\circ \tau$. Since $|\mathrm{Stab}_L(s)/(\mathrm{Stab}_L(s)\cap \mathrm{Stab}_L(t'))|=4$, we see that elements in $\mathrm{Stab}_L(s)$ map $t'$ to four other vertices $q$ in $V_j$ fulfilling $d(s,t')=d(s,q)$. Note that $s$ has four neighbors and four non-neighbors in $V_j$. Thus, there exists $\sigma_4\in \mathrm{Stab}_L(s)$ such that $\sigma_4(t')=t$. Therefore, $\sigma=\sigma_4 \circ \sigma_1$ fulfills $\sigma(k)=s$ and $\sigma(l)=t$.
\end{proof}

We will now look at products of entries in the previously defined quantum permutation matrix. The lemma inspired us to give the alternative description (Definition \ref{defqisograph}) of the graph we will discuss in the next section. 

\begin{lemma}\label{prod0}
Let $u$ be as in Lemma \ref{qpermmatrix} and let $k,s \in V_i$ and $l,t\in V_j$, $i \neq j$. 
\begin{itemize}
    \item[(i)] For $d(k,l)=d(s,t)$, we have $u_{ks}^{(i)}u_{lt}^{(j)}=0$ if and only if $\langle w_i,w_j\rangle=0$.
    \item[(ii)] For $d(k,l)\neq d(s,t)$, we have $u_{ks}^{(i)}u_{lt}^{(j)}=0$ if and only if $\langle w_i,w_j\rangle\neq0$.
\end{itemize}
\end{lemma}

\begin{proof}
We start with the proof of $(i)$. Since $d(k,l)=d(s,t)$, there exists $\sigma_M\in L$ such that $\sigma_M(k)=s$ and $\sigma_M(l)=t$ by Lemma \ref{edgepermutation}. Thus, we can choose $M_{ks}=M_{lt}:=M$ (note that $M_{kl}$ and $M_{st}$ are unique up to the stabilizers of $w_i$ and $w_j$, respectively) and get
\begin{align}\label{eq1}
    u_{ks}^{(i)}u_{lt}^{(j)}=MP_{w_i}M^*MP_{w_j}M^*=MP_{w_i}P_{w_j}M^*,
\end{align}
since we have $M^*M=1$ for all $M$ associated to the permutations in $L$. Since $P_{w_i}$ and $P_{w_j}$ are the rank-$1$ projections associated to $w_i$ and $w_j$, respectively, we have $\langle w_i,w_j \rangle=0$ if and only if $P_{w_i}P_{w_j}=0$. We see that the latter is equivalent to $u_{ks}^{(i)}u_{lt}^{(j)}=0$ by multiplying $M^*$ from the left and $M$ from the right in \eqref{eq1}. 

For $(ii)$, recall from Lemma \ref{qpermmatrix} that $u_{ks}^{(i)}$ and $u_{lt}^{(j)}$ are the rank-$1$ projections of the vectors associated to the vertices in $V_i$ and $V_j$, respectively. Fix $k,s\in V_i$ and $l\in V_j$. As every vertex in $V_i$ is connected to four vertices in $V_j$, every rank-$1$ projection associated to $V_i$ is orthogonal to four rank-$1$ projections associated to $V_j$. Thus, we have $u_{ks}^{(i)}u_{lt}^{(j)}=0$ for exactly four $t\in V_j$. Note that by \eqref{eq1}, we either have $u_{ks}^{(i)}u_{lt}^{(j)}=0$ for all $t$ with $d(k,l)=d(s,t)$ or it holds $u_{ks}^{(i)}u_{lt}^{(j)}=0$ for all $t$ with $d(k,l)\neq d(s,t)$. The assertion then follows from $(i)$. 
\end{proof}

\subsection[Rank 4 graph]{A rank $4$ graph from independent sets of the folded halved $8$-cube graph}

The folded halved $8$-cube graph can be described as follows. Denote by $\mathbf{1}$ the all ones vector in $\F_2^8$. The vertex set of the folded halved $8$-cube graph is the set of all pairs $\{x,\mathbf{1}+x\}$ for all $x\in \F_2^8$ with $\sum_i x_i=0$. Two vertices $\{x,\mathbf{1}+x\}$ and $\{y,\mathbf{1}+y\}$ are adjacent if and only if $x+y$ or $x+y+\mathbf{1}$ is a vector $v$ having $v_i=1$ in exactly two positions. 

We need the folded halved $8$-cube graph in the next definition. 

\begin{definition}[\cite{BIK}]\label{rank4graph}
Let $\Gamma_1$ be the following graph. Take as vertices one orbit of the independent sets of size $8$ under $\Z_2^6\times A_8$ in the folded halved $8$-cube graph, where two vertices are adjacent if the associated independent sets have two points in common. 
\end{definition}

The graph is strongly regular with parameters $(120, 56, 28, 24)$ and has rank $4$, meaning that the stabilizer $\mathrm{Stab}_{\aut(\Gamma_1)}(v)$ for a vertex $v$ has $4$ orbits (see \cite{BIK}). Note that the complement of this graph also came up in \cite[Table 2.2]{MS}, isomorphic to the graphs $G_2$ -- $G_5$. To see how this graph is related to $G_{E_8}$, we will now introduce graphs $G^{\boldsymbol{w}}$ and show that those are isomorphic to the complement of $\Gamma_1$. The definition of $G^{\boldsymbol{w}}$ comes from the relations between the entries of the quantum permutation matrix in Lemma \ref{prod0}. 

\begin{definition}\label{defqisograph}
For each $i \in \{1,\dots, 15\}$, choose a vector $w_i$ associated to one of the vertices in $V_i$, where $V_i$ are the sets of vertices as in $\eqref{vertexpartition}$. Let $\boldsymbol{w}=\{w_1, \dots, w_{15}\}$ and define the graph $G^{\boldsymbol{w}}$ as follows. We let $V(G^{\boldsymbol{w}})=V(G_{E_8})$. If $s\in V_i$, $t\in V_j$ with $\langle w_i, w_j\rangle \neq 0$, we let $(s,t)\in E(G^{\boldsymbol{w}})$ if and only if $(s,t)\in E(G_{E_8})$. If $s\in V_i$, $t\in V_j$ with $\langle w_i, w_j\rangle = 0$, we let $(s,t)\in E(G^{\boldsymbol{w}})$ if and only if $(s,t)\notin E(G_{E_8})$. 
\end{definition}

A convenient choice for the vectors $w_i$ are $e_1-e_j$, $x_{\{1,j\}}$, $x_{\emptyset}$ for $j \in \{2,\dots, 8\}$. Then, the only vectors that are orthogonal are $e_1-e_j, x_{\{1,j\}}$ and $e_1-e_j, x_{\emptyset}$ for $j \in \{2,\dots, 8\}$. We will see that we get isomorphic graphs for any choice of $w_i$. First, we look at automorphisms of $G^{\boldsymbol{w}}$. 

\begin{lemma}\label{autsG'}
Let $G^{\boldsymbol{w}}$ as in Definition \ref{defqisograph} and $L$ as in Lemma \ref{Lsubgroup}. We have $L \subseteq \aut(G^{\boldsymbol{w}})$. Furthermore, for $v_x \in V_i$ and $v_y \in V_j$, we have $(v_x,v_y) \notin E(G^{\boldsymbol{w}})$ if and only if there exists $\sigma\in L$ such that $\sigma(v_{w_i})=v_x$ and $\sigma(v_{w_j})=v_y$.
\end{lemma}

\begin{proof}
Let $\sigma_M \in L$. Since we have $V(G^{\boldsymbol{w}})=V(G_{E_8})$, we know that $\sigma_M $ is well-defined on $V(G^{\boldsymbol{w}})$ by Lemma \ref{Lsubgroup}. Let $V_i$, $i \in \{1, \dots, 15\}$ as in \eqref{vertexpartition}. By definition, we have $\sigma_M(v_x)\in V_i$ for all $v_x \in V_i$. Using this, the definition of $G^{\boldsymbol{w}}$ and $\sigma_M \in \aut(G_{E_8})$, we see that $\sigma_M \in \aut(G^{\boldsymbol{w}})$. Thus $L \subseteq \aut(G^{\boldsymbol{w}})$.

For the second assertion, first note that we have $(v_{w_i}, v_{w_j})\notin E(G^{\boldsymbol{w}})$ for every $i, j$.  We will denote by $\mathrm{Stab}_L(V_i)$ the stabilizer subgroup of all vertices in $V_i$. Let $v_x \in V_i$ and define $B:=M_{w_ix}$. Since $\sigma_{B}\in \aut(G^{\boldsymbol{w}})$, we have $(v_x, v_{Bw_j})\notin E(G^{\boldsymbol{w}})$ for all $j \neq i$. For every $\sigma_N\in \mathrm{Stab}_L(V_i)\backslash \mathrm{Stab}_L(V_j)$, we thus get a vertex $v_{NBw_j}$ with $(v_x, v_{NBw_j})\notin E(G^{\boldsymbol{w}})$. Since we have $|\mathrm{Stab}_L(V_i)/ (\mathrm{Stab}_L(V_i)\cap \mathrm{Stab}_L(V_j))|=4$, we thus get four non-neighbors of $v_x$ for every $j \neq i$. We get $56$ non-neighbors in this way and since the graph $G^{\boldsymbol{w}}$ is $63$-regular (the construction from $G_{E_8}$ does not change the number of neighbors of a vertex), those are all non-neighbors in $G^{\boldsymbol{w}}$. 
\end{proof}

\begin{cor}
The graphs $G^{\boldsymbol{w}^{(1)}}$ and $G^{\boldsymbol{w}^{(2)}}$ are isomorphic for any choice of $w_i^{(1)}\in V_i$ and $w_i^{(2)}\in V_i$, $i\in \{1,\dots, 15\}$. 
\end{cor}

\begin{proof}
Since we have $V(G^{\boldsymbol{w}^{(1)}})=V(G^{\boldsymbol{w}^{(2)}})=V(G_{E_8})$, we can partition the vertex sets into $V_i$ as in \eqref{vertexpartition}. Let $N_i:=M_{w_i^{(1)}w_i^{(2)}}$, where $M_{w_i^{(1)}w_i^{(2)}}$ is as in Lemma \ref{qpermmatrix}. We claim that the map $\varphi:V(G^{\boldsymbol{w}^{(1)}})\mapsto V(G^{\boldsymbol{w}^{(2)}})$, $v_x \in V_i\mapsto v_{N_ix}\in V_i$ is an isomorphism between $G^{\boldsymbol{w}^{(1)}}$ and $G^{\boldsymbol{w}^{(2)}}$. Let $v_x\in V_i, v_y\in  V_j$. Assume $(v_{N_ix}, v_{N_iy})\notin E(G^{\boldsymbol{w}^{(2)}})$. By Lemma \ref{autsG'} this is equivalent to the existence of $\tau \in L$ such that $\tau(v_{w_i^{(2)}})=v_{N_ix}$ and $\tau(v_{w_j^{(2)}})=v_{N_jy}$. Using $v_{N_ix}=\sigma_{N_i}(v_x)$, $\sigma_{N_i}^2=\mathrm{id}$,  $\sigma_{N_i}\circ \tau=\tau \circ \sigma_{N_i}$ and $\sigma_{N_i}(v_{w_i^{(2)}})=v_{w_i^{(1)}}$ and similar for $v_y$ and $j$, this is equivalent to $\tau(v_{w_i^{(1)}})=v_x$ and $\tau(v_{w_j^{(1)}})=v_y$. Again using Lemma \ref{autsG'}, we get those equations if and only if $(v_x,v_y)\notin E(G^{\boldsymbol{w}^{(1)}})$. Thus, $\varphi$ is an isomorphism between $G^{\boldsymbol{w}^{(1)}}$ and $G^{\boldsymbol{w}^{(2)}}$. 
\end{proof}

We will now show that $\overline{\Gamma}_1$ is isomorphic to the graph $G^{\boldsymbol{w}}$. 

\begin{theorem}\label{isographs}
The complement of the graph $G^{\boldsymbol{w}}$ is isomorphic to the graph $\Gamma_1$ described in Definition \ref{rank4graph}.
\end{theorem}

\begin{proof}
First note that the folded halved (or halved folded) $8$-cube graph is the complement of $VO_6^+(2)$, see \cite[Section 10.26]{BVM}. The latter graph has vertex set $\F_2^6$ and two vertices $x,y$ are adjacent if $Q(x+y)=0$, where $Q(z)=z_1z_2+z_3z_4+z_5z_6$. We will use an equivalent description. Take as vertex set $X^{x_1}Z^{x_2}\otimes X^{x_3}Z^{x_4}\otimes X^{x_5}Z^{x_6}$ for the Pauli matrices $X, Z$ as in \eqref{Pauli}, $x\in \{0,1\}^6$. Then two vertices are adjacent if in the product of the $3$-tensor Pauli matrices, the matrix $Y=XZ$ appears either zero or two times in the three tensor legs (we will consider the $3$-tensor Pauli matrices up to a sign). We construct the graph $\Gamma_1$ from the following $120$ cliques of size $8$ in $VO_6^+(2)$. 
Recall from \eqref{Stabs} the stabilizer subgroups of the vertices. We will denote by $\mathrm{Stab}_L(V_i)$ the stabilizer subgroup of all vertices in $V_i$ (they are the same by Lemma \ref{stabeq}). For $i \in \{1,\dots, 15\}$, let 
\begin{align}\label{clique1}
   C(i)=\{M=M_1\otimes M_2 \otimes M_3,\, M_i\in \{I,X,Y,Z\}\,|\, \sigma_M \in \mathrm{Stab}_L(V_i)\}.
\end{align}
They form cliques in $VO_6^+(2)$ as all elements in $C(i)$, and thus their products contain either $0$ or $2$ times the matrix $Y$ in the tensor legs. The remaining $105$ cliques of size $8$ that we choose are of the form
\begin{align}\label{clique2}
    NC(i)=\{NM \,|\, M \in C(i)\},
\end{align}
where $N=N_1\otimes N_2 \otimes N_3$ with $N_i\in \{I,X,Y,Z\}$ and $\sigma_N \notin \mathrm{Stab}_L(V_i)$. In \eqref{clique2}, we leave out signs in front of the tensor and consider the corresponding vertex in $V(VO_6^+(2))$ to $NM$. We get seven additional cliques for every $C(i)$ in this way. Choosing those $120$ cliques as vertices, which are adjacent if the corresponding cliques share two vertices of $VO_6^+(2)$ gives us the graph $\Gamma_1$ (Definition \ref{rank4graph}) because of the following. The subgroup $\Z_2^6$ is generated by left multiplication with elements of the form $N=N_1\otimes N_2 \otimes N_3$ with $N_i\in \{I,X,Y,Z\}$. Thus, we get from $C(i)$ to $NC(i)$ using those automorphisms. Let $q(x,v)=Q(x+v)+Q(v)+Q(x)$ for $x,v \in \F_2^6$. The subgroup $A_8$ of the automorphism group of $VO_6^+(2)$ is generated by even products of maps $t_v:\F_2^6\to \F_2^6$, $t_v(x)=x+q(x,v)v$, where $Q(v)=1$ (inducing automorphisms on $X^{x_1}Z^{x_2}\otimes X^{x_3}Z^{x_4}\otimes X^{x_5}Z^{x_6}$). One checks that one gets from $C(i)$ to $C(j)$ using those automorphisms (for example, we use $t_v\circ t_w$ with $v=(001100), w=(000011)$ to obtain $C(2)$ from $C(1)$). Since we know from \cite{BIK} that there are $120$ cliques in the orbit, the cliques in \eqref{clique1} and \eqref{clique2} are all of them. 

We will now use the same labelling for the complement of $G^{\boldsymbol{w}}$ to see that the graphs are isomorphic. Let $G^{\boldsymbol{w}}$ as in Definition \ref{defqisograph}. Relabel the vertices $v_{w_i}$ by $C(i)$ as in \eqref{clique1} and $v_x\in V_i$ by $M_{w_ix}C(i)$ as in \eqref{clique2}, where $M_{w_ix}$ is as in Lemma \ref{qpermmatrix}. In this labelling, the automorphisms $\sigma_M\in L$ send $NC(i)$ to $(MN)C(i)$. By Lemma \ref{autsG'}, we know that vertices $v_x$ and $v_y$ in the complement of $G^{\boldsymbol{w}}$ are adjacent if and only if there exists $\sigma_N\in L$ such that $\sigma_N(v_{w_i})=v_x$ and $\sigma_N(v_{w_j})=v_y$. With the new labelling, this means that vertices $AC(i)$ and $BC(j)$ are adjacent if and only if there exists $\sigma_N\in L$ such that $\sigma_N(C(i))=AC(i)$ and $\sigma_N(C(j))=BC(j)$. Looking at \eqref{Stabs}, we see that every $C(i)$ and $C(j)$ have two points in common. Since for adjacent $AC(i)$ and $BC(j)$ there exists $\sigma_N \in L$ such that $\sigma_N(C(i))=AC(i)$ and $\sigma_N(C(j))=BC(j)$, the sets $AC(i)$ and $BC(j)$ also have two points in common. We know that $\Gamma_1$ is $56$-regular, thus every clique shares two points with exactly $56$ other cliques. Since the complement of $G^{\boldsymbol{w}}$ is also $56$-regular and all neighbors share two points, we see that the existence of an automorphism $\sigma_N$ with $\sigma_N(C(i))=AC(i)$ and $\sigma_N(C(j))=BC(j)$ is equivalent to $AC(i)$ and $BC(j)$ sharing two points. Therefore, the complement of the graph $G^{\boldsymbol{w}}$ is isomorphic to the graph $\Gamma_1$. 
\end{proof}

\subsection[The quantum isomorphism]{The quantum isomorphism between $G_{E_8}$ and $\overline{\Gamma}_1$}

We will now show that $G_{E_8}$ and $\overline{\Gamma}_1$ are quantum isomorphic. This will follow from Lemma \ref{prod0} and Theorem \ref{isographs}.

\begin{theorem}\label{qisostronglyregular}
Let $G_{E_8}$ be the graph as in Definition \ref{defGE8} and $\Gamma_1$ as in Definition \ref{rank4graph}. Then $G_{E_8}$ and $\overline{\Gamma}_1$ are quantum isomorphic, non-isomorphic strongly regular graphs.  
\end{theorem}

\begin{proof}
By Theorem \ref{isographs}, we know that $\overline{\Gamma}_1$ is isomorphic to $G^{\boldsymbol{w}}$ for any choice of $\boldsymbol{w}$. To see that $G_{E_8}$ and $G^{\boldsymbol{w}}$ are quantum isomorphic, we have to show that there exists a quantum permutation matrix that fulfills $A_{G_{E_8}}u=uA_{G^{\boldsymbol{w}}}$. By Lemma \ref{lemprod0}, this is equivalent to showing $u_{ks}u_{lt}=0$ for $d(k,l)\neq d(s,t)$, $k,l \in V(G_{E_8}), s,t \in V(G^{\boldsymbol{w}})$. Let $u$ be the quantum permutation matrix as in Lemma \ref{qpermmatrix}, where we choose the same $w_i$ as in the construction of $G^{\boldsymbol{w}}$. Since $u$ has block form (see \eqref{blockmatrix}) and we know that all vertices in $V_i$ are connected in $G_{E_8}$ and $G^{\boldsymbol{w}}$, it remains to show $u_{ks}^{(i)}u_{lt}^{(j)}=0$ for $d(k,l)\neq d(s,t)$, $k,l \in V(G_{E_8}), s,t \in V(G^{\boldsymbol{w}})$ for $i \neq j$. By definition of $G^{\boldsymbol{w}}$, we thus have to show $u_{ks}^{(i)}u_{lt}^{(j)}=0$ for $d(k,l)\neq d(s,t)$, $k,l,s,t \in V(G_{E_8})$ if $\langle w_i,w_j \rangle \neq0$ and $u_{ks}^{(i)}u_{lt}^{(j)}=0$ for $d(k,l)= d(s,t)$, $k,l,s,t \in V(G_{E_8})$  for $\langle w_i,w_j \rangle =0$. But this follows from Lemma \ref{prod0}. 

The graphs $G_{E_8}$ and $G^{\boldsymbol{w}}$ are not isomorphic because of the following. It is well known that $G_{E_8}$ has independence number $8$ (see for example \cite[Section 10.39]{BVM}). By construction, the set of vertices $\{v_{w_1}, \dots, v_{w_{15}}\}$ is an independent set of size $15$ in $G^{\boldsymbol{w}}$, which already shows $G_{E_8}\ncong G^{\boldsymbol{w}}$. 
Thus, $G_{E_8}$ and $\overline{\Gamma}_1$ are quantum isomorphic, non-isomorphic strongly regular graphs.
\end{proof}

\begin{remark}\label{remqiso}
\begin{itemize}
\item[(i)] Note that deleting a clique in a graph decreases the independence number by at most one. Let $T\subseteq \{1,\dots, 15\}$ and $V_T=\bigcup_{i\in T}V_i$, where $V_i$ is the partition of $V(G_{E_8})$ as before (see \eqref{vertexpartition}). For $|T|\geq 9$, the induced subgraphs of $G_{E_8}$ and $G^{\boldsymbol{w}}$ on $V_T$ are still quantum isomorphic and non-isomorphic. They are quantum isomorphic, since the submatrix on $V_T$ of $u$ (see \eqref{blockmatrix})  is still a quantum permutation matrix fulfilling the required relations. They are non-isomorphic, since the independence number of the subgraph of $G_{E_8}$ is less or equal to $8$ whereas the independence number of the subgraph of $G^{\boldsymbol{w}}$ is bigger or equal to $9$. 
\item[(ii)] It is noted in \cite[Section 10.39]{BVM}, that $G_{E_8}$ is distance-transitive. Thus, we know that both the quantum orbital algebra and the orbital algebra are $3$-dimensional. By \cite{BIK}, we know that $\overline{\Gamma}_1$ is a rank $4$ graph, which means that the orbital algebra is $4$-dimensional. Since $G_{E_8}$ and $\overline{\Gamma}_1$ are quantum isomorphic, we know that the quantum orbital algebra of $\overline{\Gamma}_1$ is $3$-dimensional (see \cite[Theorem 4.6]{LMR}). Thus, the graph $\overline{\Gamma}_1$ is the first example of a graph that has $3$-dimensional quantum orbital algebra and $4$-dimensional orbital algebra. For more on quantum orbital algebras, see \cite{LMR}. 
\item[(iii)] Looking at \cite[Table 2.2]{MS}, we see that $|\aut(G_{E_8})|=348364800$ and $|\aut(\overline{\Gamma}_1)|=1290240$. Thus, we have an example of quantum isomorphic graphs, where the sizes of the automorphism groups differ. This also implies that $\overline{\Gamma}_1$ has quantum symmetry, since the quantum automorphism groups of quantum isomorphic graphs are monoidally equivalent \cite[Theorem 4.7]{BCEHPSW}.
\end{itemize}
\end{remark}

\section{Switching quantum isomorphic graphs}\label{secswitching}

In this section, we will use Godsil-McKay switching to construct more quantum isomorphic, non-isomorphic graphs from the pair we obtained in the last section. We start with the definition of Godsil-McKay switching.

\begin{definition}[Godsil-McKay switching, \cite{GM}]\label{GMswitching}
Let $G$ be a graph and $\pi=\{C_1, \dots, C_k, D\}$ be a partition of the vertex set $V(G)$. Suppose that we have for $1 \leq i,j \leq k$, $v \in D$
\begin{itemize}
    \item[(i)] any two vertices in $C_i$ have the same number of neighbors in $C_j$,
    \item[(ii)] the vertex $v$ has either $0$, $\frac{n_i}{2}$ or $n_i$ neighbors in $C_i$, where $n_i:=|C_i|$.
\end{itemize}
The graph $G^{\pi, D}$ is obtain as follows. For each $v\in D$ and $1 \leq i \leq k$ such that $v$ has $\frac{n_i}{2}$ neighbors in $C_i$, delete these $\frac{n_i}{2}$ edges and join $v$ instead to the other $\frac{n_i}{2}$ vertices in $C_i$. 

Let $Q_m=\frac{2}{m}J_m-1_{M_m(\C)}$, where $J_m$ is the all-ones matrix. In terms of the adjacency matrix, we have $A_{G^{\pi, D}}=QA_GQ$, where $Q$ is the following matrix
\begin{align}
\begin{blockarray}{ccccccc}
&C_1&C_2&\dots&\dots&C_k&D\\
\begin{block}{c(cccccc)}
C_1&Q_{n_1}&0&0&\dots&\dots&0\\ C_2&0&Q_{n_2}&0&\dots&\dots&0\\ \vdots&0&0&Q_{n_3}&\dots&\dots&0\\\vdots&\vdots &\vdots&\vdots&\ddots&\dots&0\\
C_k&\vdots &\vdots&\vdots&\vdots&Q_{n_k}&0\\
D&0&0&0&0&0&1_{M_{|D|}(\C)} \\
\end{block}
\end{blockarray}\label{switchingmatrix}
\end{align}
\end{definition}

\begin{theorem}[\cite{GM}]
The graphs $G$ and $G^{\pi, D}$ are cospectral.
\end{theorem}

The next theorem shows that we can find new pairs of quantum isomorphic, non-isomorphic graphs from certain vertex partitions compatible with the block structure of a quantum permutation matrix associated to a quantum isomorphism between the graphs. 

\begin{theorem}\label{thm:qisoGM}
Let $G_1$ and $G_2$ be quantum isomorphic graphs, where there exists a quantum permutation matrix $u$ with $uA_{G_1}=A_{G_2}u$ of the form
\begin{align}
\begin{blockarray}{cccccc}
&V_1&V_2&\dots&\dots&V_m\\
\begin{block}{c(ccccc)}
V_1&u^{(1)}&0&0&\dots&0\\ V_2&0&u^{(2)}&0&\dots&0\\ \vdots&0&0&u^{(3)}&\dots&0\\\vdots&\vdots &\vdots&\vdots&\ddots&0\\
V_m&0&0&0&0&u^{(m)}\\
\end{block}
\end{blockarray}\label{switchqpermmatrix}
\end{align}
for some partition $\{V_1, \dots, V_m\}$  of the vertex sets (we can label both vertex sets by $V$ as quantum isomorphic graphs have the same number of vertices). Let $\{S_1, \dots, S_{k+1}\}$ be a partition of $[m]$ and define a partition $\pi=\{C_1, \dots, C_k, D\}$ of the vertex set by setting $C_i:=\cup_{s\in S_i}V_j$ and $D:=\cup_{s\in S_{k+1}}V_j$.
If $G_1$ and $G_2$ fulfill the properties $(i)$ and $(ii)$ of Definition \ref{GMswitching} with respect to $\pi$, we can use Godsil-McKay switching and the graphs $G_1^{\pi, D}$ and $G_2^{\pi, D}$ are quantum isomorphic. 
\end{theorem}

\begin{proof}
We first show $uQ=Qu$ for $Q$ as in \eqref{switchingmatrix} and $u$ as in \eqref{switchqpermmatrix}. By the block form of $Q$ and $u$ as well as reordering the $V_i$'s, we have $uQ=Qu$ if and only if $(\bigoplus_{s\in S_i}u^{(i)})Q_{n_i}=Q_{n_i}(\bigoplus_{s\in S_i}u^{(i)})$. Since the matrices $u_{S_i}:=\bigoplus_{s\in S_i}u^{(i)}$ are also quantum permutation matrices, we compute
\begin{align*}
    (u_{S_i}Q_{n_i})_{kl}=\frac{2}{n_i}\sum_j (u_{S_i})_{kj}- (u_{S_i})_{kl}=\frac{2}{n_i}-(u_{S_i})_{kl}=\frac{2}{n_i}\sum_j (u_{S_i})_{jl}-(u^{(i)})_{kl}=(Q_{n_i}u_{S_i})_{kl}, 
\end{align*}
where we used $\sum_j (u_{S_i})_{kj}=1=\sum_j (u_{S_i})_{jl}$. 

We have $uA_{G_1}=A_{G_2}u$ by assumption and know $A_{G_a^{\pi, D}}=QA_{G_a}Q$ for $a=1,2$ by Definition \ref{GMswitching}. Using this, we deduce
\begin{align*}
    uA_{G_1^{\pi, D}}=uQA_{G_1}Q=QA_{G_2}Qu=A_{G_2^{\pi, D}}u. 
\end{align*}
Thus $G_1^{\pi, D}$ and $G_2^{\pi, D}$ are quantum isomorphic. 
\end{proof}

Finally, we give the following example. 

\begin{example}
Let $G_{E_8}$ be the graph as in Definition \ref{defGE8} and $G^{\boldsymbol{w}}$ as in Definition \ref{defqisograph}, where we choose the vectors $w_i$ as $e_1-e_j$, $x_{\{1,j\}}$, $x_{\emptyset}$ for $j \in \{2,\dots, 8\}$. Choose the vertex set partition $\pi=\{V_1, \dots, V_{15}\}$ as in \eqref{vertexpartition} and let $D=V_{15}$. Then, for $G_{E_8}$ and $G^{\boldsymbol{w}}$, we know that 
\begin{itemize}
    \item[(i)] vertices in $V_i$ have $4$ neighbors in $V_j$,
    \item[(ii)] every $v\in D=V_{15}$ has $4=\frac{|V_i|}{2}$ neighbors in $V_i$. 
\end{itemize}
It is easy to see that $\pi$ satisfies the conditions in Theorem \ref{thm:qisoGM}. Therefore, we get that $(G_{E_8})^{\pi, V_{15}}$ and $(G^{\boldsymbol{w}})^{\pi, V_{15}}$ are quantum isomorphic. The graphs are non-isomorphic because of the following. Recall that we have $\alpha(G_{E_8})=8$ and $\alpha(G^{\boldsymbol{w}})=15$. Since $D$ is a clique, switching can at most increase or decrease the independence number by one, as at most one vertex in $D$ can be part of an independent set. Thus, we get $\alpha((G_{E_8})^{\pi, V_{15}})\leq 9$ and $\alpha((G^{\boldsymbol{w}})^{\pi, V_{15}})\geq 14$, which yields that the graphs are non-isomorphic. 

Using Sage (\cite{sagemath}), we see that the graphs $(G_{E_8})^{\pi, V_{15}}$ and $(G^{\boldsymbol{w}})^{\pi, V_{15}}$ are non-isomorphic to $G_{E_8}$ and $G^{\boldsymbol{w}}$. Since $(G_{E_8})^{\pi, V_{15}}$ and $(G^{\boldsymbol{w}})^{\pi, V_{15}}$ are cospectral to $G_{E_8}$ and $G^{\boldsymbol{w}}$, respectively, the graphs $(G_{E_8})^{\pi, V_{15}}$ and $(G^{\boldsymbol{w}})^{\pi, V_{15}}$ are strongly regular with parameters $(120, 63, 30, 36)$. 
\end{example}

\begin{remark}
\begin{itemize}
    \item[(i)] We do not know how many different pairs of quantum isomorphic, non-isomorphic strongly regular graphs we get by using Godsil-McKay switching several times. 
    \item[(ii)] Similar to Remark \ref{remqiso} $(i)$, we can delete cliques and obtain more quantum isomorphic, non-isomorphic graphs in that way. 
\end{itemize}
\end{remark}

\paragraph{Acknowledgments.}\phantom{a}\newline
The author has received funding from the European Union's Horizon 2020 research and innovation programme under the Marie Sklodowska-Curie grant agreement No. 101030346. He thanks David Roberson for helpful discussions on quantum isomorphisms and graph switching. 
\bibliographystyle{plainurl}
\bibliography{qisostronglyregular}

\end{document}